%% file: nc-hull.tex
\documentclass{amsart}

\usepackage{amsmath,amsthm,amssymb}
\usepackage[hidelinks]{hyperref}
\usepackage{tikz,graphicx}
\usepackage{standalone}

\graphicspath{
  {./}
  {figures/}
}

\usetikzlibrary{calc}
\tikzstyle{dot}=[shape=circle,draw,color=black,fill=black,inner sep=1.5pt]
\input{tikz-styles.tex}

\theoremstyle{plain}
\newtheorem{thm}{Theorem}[section]
\newtheorem{lem}[thm]{Lemma}

\newtheorem{prop}[thm]{Proposition}
\newtheorem*{conj}{Conjecture}

\newtheorem{mainthm}{Theorem}

\theoremstyle{definition}
\newtheorem{defn}[thm]{Definition}

\newcommand{\NC}{\textsc{NC}}
\newcommand{\CC}{\mathcal{C}}
\newcommand{\HC}{\mathcal{H}}

\newcommand{\conv}{\textsc{Conv}}
\newcommand{\bool}{\textsc{Bool}}
\newcommand{\cat}{{\operatorfont{CAT}}}
\newcommand{\conf}{{\operatorfont{Conf}}}

\newcommand{\upper}{{\uparrow}}
\newcommand{\low}{{\downarrow}}
\newcommand{\rk}{{\operatorfont rk}}
\newcommand{\shape}{\textsc{Shape}}
\newcommand{\partition}{\textsc{Part}}

\title{Noncrossing partitions from hull configurations}
\author{Michael Dougherty and Gina Root}
\date{\today}

\begin{document}

\begin{abstract}
    Each finite configuration of points in the plane determines a 
    corresponding lattice of noncrossing partitions. When these 
    points form the vertex set of a convex polygon, the associated
    lattice is the classical noncrossing partition lattice (introduced
    by Kreweras in 1972), which
    makes many appearances in combinatorics and geometric group theory.
    If, on the other hand, all points of the configuration lie on
    a common line segment, the result is a Boolean lattice. In this
    article, we examine the more general class of 
    \emph{hull configurations}, which we define
    to be those which lie either on a line segment or on the
    boundary of a convex polygon.
    We prove that 
    the corresponding lattices of noncrossing partitions are unions
    of maximal Boolean subposets and, under certain circumstances,
    have symmetric chain decompositions.
\end{abstract}

\maketitle

\section*{Introduction}

If $P$ is the vertex set of a convex $n$-gon in the plane, then a
set partition of $P$ is \emph{crossing} if there are two parts 
of the partition with overlapping convex hulls, or \emph{noncrossing}
otherwise. The \emph{classical lattice of noncrossing partitions} $\NC(n)$ 
is the set of all noncrossing partitions of $P$, partially ordered by 
refinement. This poset was introduced by Kreweras in 1972 \cite{kreweras72} and 
has since become an important object in areas such as algebraic 
combinatorics, geometric group theory, and free probability. See \cite{baumeister19}
and \cite{mccammond06} for surveys.

By choosing a different starting configuration $P$ (which is not necessarily
the vertex set of a convex polygon) and considering the
corresponding poset of noncrossing partitions $\NC(P)$, one obtains 
a natural generalization of $\NC(n)$. As in the classical case, $\NC(P)$
is a lattice for every choice of $P$ \cite{cdhm24}, but other properties vary. 
For example, recent work of the authors with Fang, Jiang, Lin, Lindenmuth and Pokras
\cite{dfjllpr25} shows that when the configuration $P$ lies on either a semicircle 
or a pair of rays with a common origin, then the resulting lattice $\NC(P)$ is 
graded and rank-symmetric; in particular, it has a symmetric chain decomposition.

In this article, we examine a larger class of configurations which includes the 
cases studied in \cite{dfjllpr25}. In brief, we define a \emph{hull configuration}
to be one which lies either on a line segment or on the boundary of a convex $n$-gon.
Intuitively, the noncrossing partition lattices for these configurations interpolate 
between the Boolean lattice (i.e. the lattice of noncrossing partitions for a configuration
which lies on a line segment) and the classical lattice of noncrossing partitions. 
We also define a poset $\HC(n)$ of all hull configurations with $n$ points, 
up to an equivalence relation which allows for deformations of the configuration
without disturbing collinearities. Minimal elements of $\HC(n)$ are those which
lie on a line segment, and maximal elements are those which lie on the vertex set of
a convex $n$-gon. Each maximal chain in $\HC(n)$ consists of $n$ equivalence classes
of configurations, which can be viewed as beginning with the vertex set of a polygon,
then introducing collinearities one at a time until the resulting configuration
lies on a line segment.

Our first main theorem extends the aforementioned result on symmetric chain 
decompositions to hull configurations $P$ which have a ``blank side,'' i.e. an edge of the 
convex hull which contains only two points of $P$.
Boolean lattices were shown to have symmetric chain decompositions
in a 1951 article by de Bruijn, van Ebbenhorst Tengbergen and
Kruyswijk \cite{de-bruijn}, and the same was shown for the classical noncrossing
partition lattice $\NC(n)$ by Simion and Ullman in 1999 \cite{simion-ullman91}.
Rephrased using our terminology, these articles showed that $\NC(P)$ admits
a symmetric chain decomposition when $P$ is either a minimal element of
$\HC(n)$ (in which case $\NC(P)$ is isomorphic to $\bool(n-1)$) or
a maximal element of of $\HC(n)$ (in which case $\NC(P)$ is isomorphic to 
$\NC(n)$). In \cite{dfjllpr25}, this result was extended to hull configurations
referred to as cones and semicircles; the present article provides an even 
greater generalization.

\begin{mainthm}[Theorem~\ref{thm:scd}]
    \label{mainthm:scd}
    If $P$ is a hull configuration with at least one blank side, 
    then $\NC(P)$ has a symmetric chain decomposition.
\end{mainthm}

Based on several examples of hull configurations with no blank sides,
it seems likely that Theorem~\ref{mainthm:scd} can be strengthened, as
we record in the following conjecture.

\begin{conj}
    If $P$ is a hull configuration with no blank sides, then
    $\NC(P)$ is not rank-symmetric (and therefore does not have
    a symmetric chain decomposition).
\end{conj}

For example, suppose that $P$ is a configuration of six points such that
three points lie on the vertices of a triangle and the other three points
lie on the midpoints of the triangle's edges. Then $\NC(P)$ is a
graded lattice with 1 element of rank 0, 12 elements of rank 1, 
34 elements of rank 2, 35 elements of rank 3, 12 elements of rank 4, and
1 element of rank 5. Thus $\NC(P)$ is not rank-symmetric.

Our second main theorem concerns the maximal Boolean subposets of $\NC(P)$.
In the classical case, it was previously shown by T. Brady and McCammond
that $\NC(n)$ is a union of maximal Boolean subposets \cite{bm10}. 
In particular, it was shown by Haettel, Kielak and Schwer \cite{hks16} 
and McCammond \cite{hypertrees} that the maximal Boolean subposets of $\NC(n)$ are 
in one-to-one correspondence with noncrossing trees with $n$ vertices. 
More information on this correspondence and its relationship to the theory of
buildings can be found in the work of Heller and Schwer \cite{heller-schwer},
and in Heller's doctoral dissertation \cite{heller-thesis}.

In our more general setting, we 
say that a tree $\tau$ with vertex set $P$ has \emph{convex geodesics} if, for 
each side of $\conv(P)$, the unique geodesic in $\tau$ between the two endpoints
of that side passes through every other point of $P$ which lies in the interior
of that side (Definition~\ref{def:convex-geodesics}).

\begin{mainthm}[Theorem~\ref{thm:boolean-bijection} and Theorem~\ref{thm:boolean-union}]
    \label{mainthm:boolean}
    If $P$ is a hull configuration, then $\NC(P)$ is a union of maximal 
    Boolean subposets, which are in one-to-one correspondence with the
    noncrossing trees on $P$ with convex geodesics.
\end{mainthm}

In particular, this result has an interesting connection with the intrinsic
geometry of the $n$-strand braid group. If we let $\overline{\NC(n)}$ denote
the classical lattice of noncrossing partitions $\NC(n)$ with its maximum and 
minimum elements removed, then the order complex $|\overline{\NC(n)}|_{\Delta}$
(also known as the ``diagonal link'' of $\NC(n)$) is an $(n-3)$-dimensional 
ordered simplicial complex whose $k$-dimensional simplices correspond to chains 
in $\overline{\NC(n)}$ with $k+1$ elements. It was shown in \cite{bm10} that, 
with the appropriate metric, $|\overline{\NC(n)}|_{\Delta}$ embeds in a 
spherical building in such a way that the image is a union of apartments.
Moreover, Brady and McCammond conjectured that $|\overline{\NC(n)}|_{\Delta}$
is a $\cat(1)$ metric space, which would imply that the $n$-strand braid group
is a $\cat(0)$ group. This has been proven when $n \leq 7$ 
\cite{bm10,hks16,jeong23}, but remains open in general.

In this context, Theorem~\ref{mainthm:boolean} tells us that when $P$ is a
hull configuration, the diagonal link of $\NC(P)$ is a subcomplex of
the diagonal link for $\NC(n)$ which can also be expressed as a union of
apartments. It would be 
interesting to know whether the smaller diagonal links found in $\NC(P)$ could 
be used to better understand the diagonal link for $\NC(n)$, and thus the 
intrinsic geometry of the braid group. 

The article is organized into three sections. The first introduces
our definitions for hull configurations and their associated partial order.
The second section reviews symmetric chain decompositions and proves 
Theorem~\ref{mainthm:scd}; the third section includes the proof of
Theorem~\ref{mainthm:boolean}.

\section{Convexity classes and hull configurations}
\label{sec:convexity-classes}

In this section we introduce some basic terminology and properties
for configurations. To start, each configuration of $n$ points
in $\mathbb{C}$ can be viewed as a single point in the 
\emph{ordered configuration space} $\conf_n(\mathbb{C})$,
the topological subspace of $\mathbb{C}^n$ consisting of all $n$-tuples 
with distinct entries. Since we will primarily be interested in
appearances of collinearities (rather than distances between points,
for example), we introduce a useful equivalence relation on
$\conf_n(\mathbb{C})$.

\begin{defn}
    \label{def:convexity-class}
    For each configuration $P$ in $\conf_n(\mathbb{C})$, 
    let $\mathcal{L}(P)$ be the arrangement of
    lines in $\mathbb{C}$ which pass through pairs of points in $P$
    and observe that $1\leq |\mathcal{L}(P)| \leq \binom{n}{2}$.
    If a line in $\mathcal{L}(P)$ contains more than two points
    in $P$, then we say that the two extreme points on the line are
    \emph{external collinearities} and the other points on the line
    are \emph{internal collinearities}.
    For each $z \in P$, define the \emph{multiplicity} 
    $\ell(z)$ to be the number of lines
    for which $z$ is an internal collinearity. 
    Define the function $L \colon \conf_n(\mathbb{C}) \to \mathbb{Z}^n$ by 
    $L(z_1,\ldots,z_n) = (\ell(z_1),\ldots,\ell(z_n))$ and for each $P \in \conf_n(\mathbb{C})$,
    let $X_P$ denote the path component of $L^{-1}(L(P)) \subset \conf_n(\mathbb{C})$ containing $P$.
    Finally, we define an equivalence relation on $\conf_n(\mathbb{C})$
    by declaring $P\sim Q$ if $X_P = X_Q$.
    We refer to the equivalence classes for this relation as
    \emph{convexity classes}. For the remainder of the article, we will 
    not distinguish between a configuration and the convexity class 
    which contains it.
\end{defn}

\begin{figure}
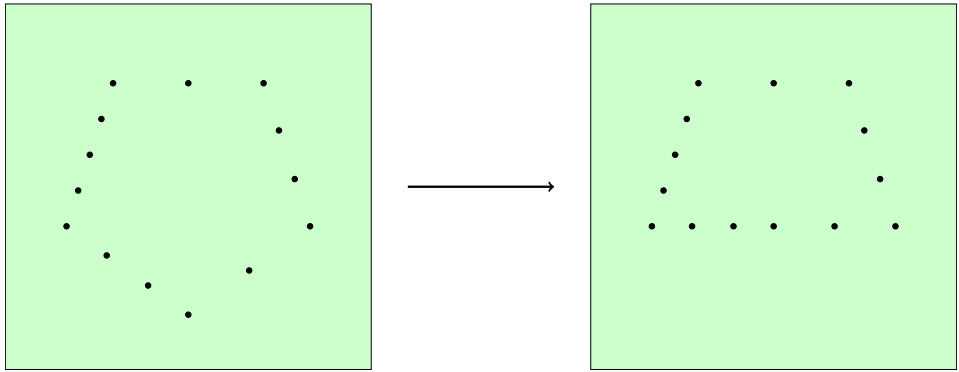

    \centering
    \includestandalone[width=\textwidth]{fig-elem-collapse}
    \caption{An elementary collapse $m$ from a hull
    configuration $Q$ (left) to another hull configuration $m(Q) = P$ (right)}
    \label{fig:elem-collapse}
\end{figure}

The usual total order on the integers leads to a partial order on
$\mathbb{Z}^n$ in the following manner: 
$(a_1,\ldots,a_n) \leq (b_1,\ldots,b_n)$ if $a_i \leq b_i$
for all $i$. Note that this is a graded lattice with rank function
$\rk \colon \mathbb{Z}^n \to \mathbb{Z}$ given by 
$\rk(a_1,\ldots,a_n) = a_1+\cdots+a_n$.
This partial order can be pulled back through the map $L$
to define a partial order on the set of convexity classes.

\begin{defn}
    \label{def:poset-of-conv-classes}
    Let $\CC(n)$ denote the set of all convexity classes 
    with $n$ points. Use the typical partial order on $\mathbb{Z}^n$ 
    to define a partial order on $\CC(n)$ as follows:
    declare $P\leq Q$ in $\CC(n)$ if $L(Q) \leq L(P)$ in $\mathbb{Z}^n$ and the 
    path component $X_P$ is contained within the closure of the 
    path component $X_Q$. In other words, $P \leq Q$ if
    $Q$ can be continuously deformed into $P$ by (potentially) 
    introducing additional collinearities without removing any. 
    When $P < Q$ is a covering relation (i.e. there is no $R \in \CC(n)$
    with $P < R < Q$), we say that the deformation transforming $Q$ into $P$ is an \emph{elementary collapse}; see Figure~\ref{fig:elem-collapse}.
\end{defn}

It is straightforward to see from the preceding definition that
$\CC(n)$ is a graded poset of height $n-2$ with rank function 
$\rk \colon \CC(n) \to \mathbb{Z}$ given by defining
\[\rk(z_1,\ldots,z_n) = n-(\ell(z_1) + \cdots + \ell(z_n)).\]
Note that the minimal elements of $\CC(n)$ have rank $2$ and
the maximal elements have rank $n$. In particular, every
element of $\CC(n)$ can be obtained from a maximal element
via a finite sequence of elementary collapses.

Next, we introduce a special type of
convexity class which forms the main object of study in this article.

\begin{defn} 
    \label{def:hull-configuration}
    We say that $P \in \CC(n)$ is a \emph{hull configuration}
    if the convex hull $\conv(P)$ is either one-dimensional or
    contains no elements of $P$ in its interior. 
    Observe that for each point $z$ in a hull configuration $P$, 
    the multiplicity $\ell(z)$ is either $0$ or $1$.
    Let $\HC(n)$ denote the subposet of all hull configurations
    in $\CC(n)$, and note that $\HC(n)$ is a graded poset with the 
    same rank function 
    as $\CC(n)$. In fact, the rank function has an alternative definition in this 
    setting: for all $P \in \HC(n)$, $\rk(P) = 2$ if the convex hull $\conv(P)$ 
    is one-dimensional, and otherwise $\rk(P) = k$, where $k$
    is the number of vertices in $\conv(P)$.
\end{defn}

To record the number of internal collinearities
in a hull configuration, we introduce the following notation.

\begin{defn}
    \label{def:shape-of-conf}
    Let $P$ be a hull configuration of rank $k\geq 3$ in $\HC(n)$. 
    Number the sides of $\conv(P)$ from $1$ to $k$ counterclockwise, 
    and let $c_i$ denote the number of internal collinearities 
    on the side numbered $i$. We define the \emph{shape} of $P$
    to be $\shape(P) = [c_1; \ldots; c_k]$, the
    equivalence class of the $k$-tuple $(c_1,\ldots,c_k)$ up to 
    cyclic shifts. In other words, $\shape(P)$ is a cyclic composition
    of $n-k$ into $k$ parts (some of which may be zero). When $c_i=0$, we say that side $i$ is a \emph{blank side} of the convex hull of $P$. Finally, we label the
    $c_i+2$ points of $P$ on side $i$ as $z_{i,0}, z_{i,1}, \ldots, z_{i,c_i}, z_{i+1,0}$,
    where this linear order on side $i$ comes from the counterclockwise cyclic order 
    on $P$, and the first subscript is evaluated mod $k$. See
    Figure~\ref{fig:conf-shape} for an illustration.
\end{defn}

\begin{figure}
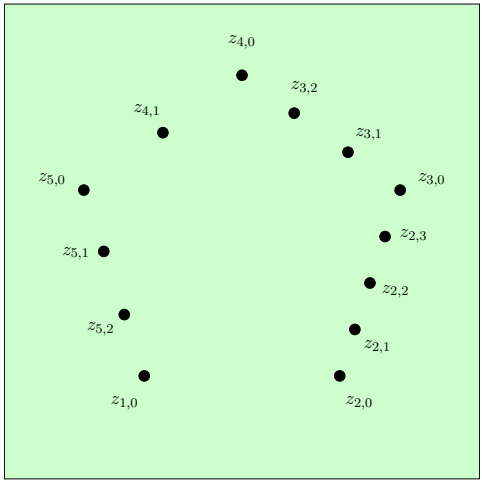

    \centering
    \includestandalone[width=0.5\textwidth]{fig-configuration-example}
    \caption{An illustration of our labeling conventions for
    a hull configuration with shape $[0; 3; 2; 1; 2]$.}
    \label{fig:conf-shape}
\end{figure}

One can quickly see that
$\HC(n)$, like $\CC(n)$, has $n!/2$ minimal elements, each of which
corresponds to a unique way of linearly ordering $n$ points up to reversal.
In higher ranks, the number of elements is counted as follows.

\begin{prop}
    Let $3\leq k\leq n$. The number of elements of rank $k$ in $\HC(n)$ is 
    $(n-1)! \binom{n}{k}$. In particular, $\HC(n)$ has $(n-1)!$ maximal
    elements.
\end{prop}

\begin{proof}
    A hull configuration of rank $k$ is uniquely determined by two
    independent choices: one of the $(n-1)!$ different cyclic orders for
    the $n$ points, and one of the $\binom{n}{k}$ different ways to
    choose $k$ vertices from those $n$ points.
\end{proof}

This poset also satisfies some other interesting 
enumerative formulas, presented in the following proposition.
Recall that for an element $x$ in a poset $R$, the 
\emph{lower set} is the subposet $\low(x) = \{y \in R \mid y \leq x\}$ and the
\emph{upper set} is the subposet $\upper(x) = \{y \in R \mid x \leq y\}$.
See Figures~\ref{fig:cc-lower-set} and \ref{fig:cc-upper-set} for examples
in $\CC(4)$.

\begin{figure}
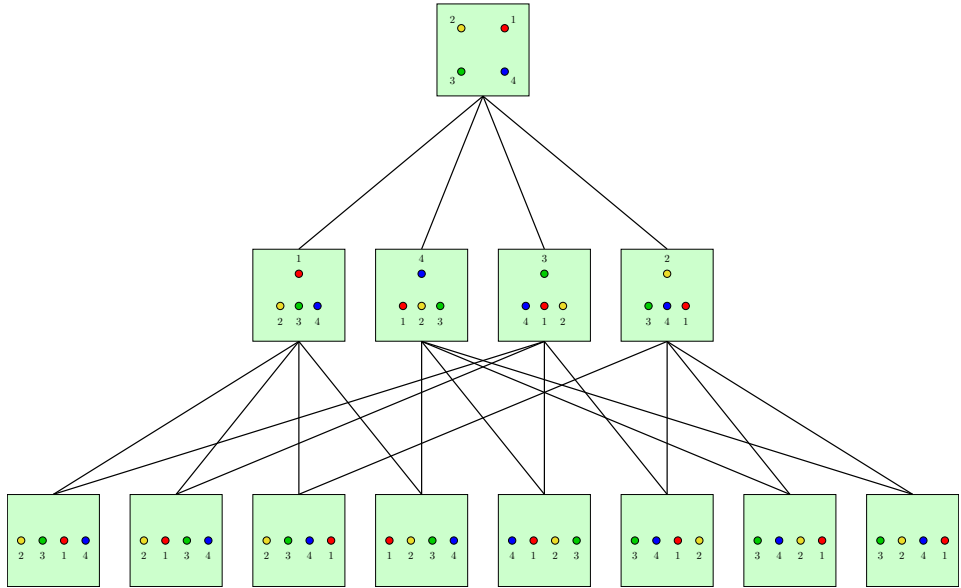

    \centering
    \includestandalone[width=\textwidth]{fig-cc-lower-set}
    \caption{The lower set for a maximal element of $\CC(4)$, the poset of 
    convexity classes with four points.}
    \label{fig:cc-lower-set}
\end{figure}

\begin{prop}
    \label{prop:hc-enum-upper-lower-sets}
    Each minimal element of $\HC(n)$ has $2^{n-2}$ maximal elements
    above it; each maximal element of $\HC(n)$ has $n2^{n-3}$ minimal
    elements below it.
\end{prop}

\begin{proof}
    For the first claim, observe that each minimum element of $\HC(n)$
    is a configuration $P$ of $n$ points, of which two are external and $n-2$ 
    are internal collinearities. If we imagine these points arranged on
    a horizontal line segment, then each maximal configuration in $\HC(n)$
    above $P$ is determined by a choice of moving each internal collinearity
    either above or below the line. Since these choices can be made 
    independently of one another, there are $2^{n-2}$ maximal elements 
    above $P$ in $\HC(n)$. 

    To see the second claim, recall that $\HC(n)$ has $(n-1)!$ maximal 
    elements and $n!/2$ minimal elements. By symmetry, this means that
    the number of minimal elements below a given maximum of $\HC(n)$ is
    equal to $(n!/2)2^{n-2}/(n-1)!$, which simplifies to $n2^{n-3}$.
\end{proof}

\begin{figure}
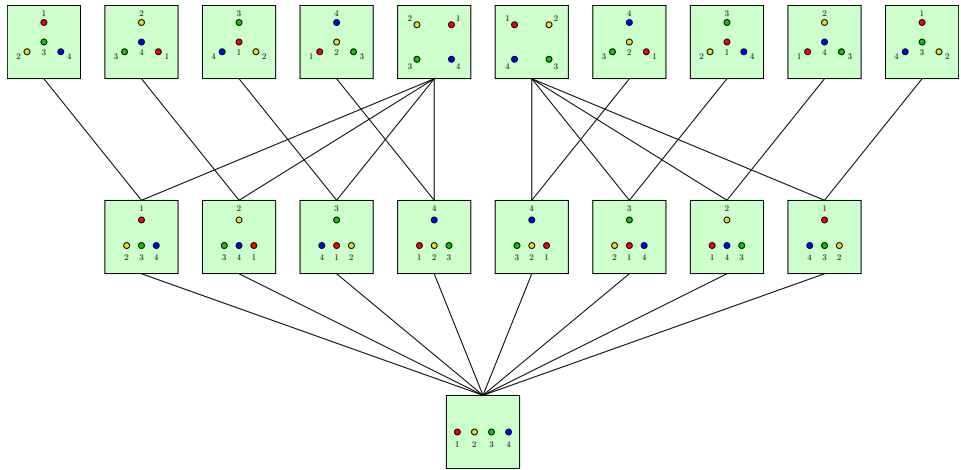

    \centering
    \includestandalone[width=\textwidth]{fig-cc-upper-set}
    \caption{The upper set for a minimal element of $\CC(4)$, the poset of 
    convexity classes with four points. Note that only two of the ten maximal
    elements also belong to $\HC(4)$, the poset of hull configurations
    with four points.}
    \label{fig:cc-upper-set}
\end{figure}

We close this section by noting that intervals in $\HC(n)$ have 
a familiar structure.

\begin{prop}
    If $P\leq Q$ in $\HC(n)$, then the interval $[P,Q]$
    is isomorphic to a Boolean lattice of height $rk(Q) - rk(P)$.
\end{prop}

\begin{proof}
    Suppose that $P\leq Q$ and $rk(Q) - rk(P) = k$. Then there
    is a motion which takes $Q$ to $P$ while preserving all existing
    internal collinearities in $Q$ and adding $k$ new ones. These
    new collinearities can be added independently in any order, and
    so each element of $[P,Q]$ in $\HC(n)$ is completely determined 
    by a subset of these $k$ internal collinearities. Thus 
    the interval $[P,Q]$ is isomorphic to $\bool(k)$,
    the poset of subsets of a $k$-element set.
\end{proof}

\section{Symmetric chain decompositions and rank-symmetry}
\label{sec:scd}

In this section we introduce the poset of noncrossing partitions
associated to a convexity class and examine some its properties.
In particular, we show that the lattice $\NC(P)$ has a symmetric
chain decomposition when the configuration $P$ has a blank side,
proving Theorem~\ref{mainthm:scd}. First, we give a 
definition for $\NC(P)$.

\begin{defn}
    For each $P \in \conf_n(\mathbb{C})$, a set partition
    of $P$ into subsets is \emph{noncrossing} if the convex
    hulls of the subsets are pairwise disjoint. The set of all
    partitions for $P$, denoted $\Pi(P)$, forms a lattice under
    the partial order of refinement, and the subset of all
    noncrossing partitions forms a subposet denoted $\NC(P)$.
    Note that there is a natural isomorphism between $\NC(P)$ and
    $\NC(Q)$ whenever $P$ and $Q$ belong to the same
    convexity class, so it is safe to define the lattice
    $\NC(P)$ for each $P \in \CC(n)$.
\end{defn}

As described in the introduction, $\NC(P)$ is isomorphic to
the classical lattice of noncrossing partitions $\NC(n)$ when
$P$ is a maximal element of $\HC(n)$. 
Next, we note that the partial order on $\CC(n)$ corresponds to 
inclusion of the corresponding lattices of noncrossing partitions.

\begin{prop}
    \label{prop:nc-subposets}
    If $P\leq Q$ in $\CC(n)$, then $\NC(P)$ is isomorphic to a subposet
    of $\NC(Q)$. In particular, if $P \in \HC(n)$, then $\NC(P)$ is
    isomorphic to a subposet of the classical noncrossing partition lattice
    $\NC(n)$.
\end{prop}

\begin{proof}
    If $P\leq Q$ in $\CC(n)$, then $P$ can be obtained from $Q$ by a
    continuous motion $m\colon Q \to P$ which preserves all existing collinearities 
    and potentially adds some new ones. This induces an isomorphism of
    partition lattices $m \colon \Pi(Q) \to \Pi(P)$, and no crossing partition
    of $Q$ can become noncrossing by adding collinearities, which means
    that restricting the domain of $m$ to $\NC(Q)$ yields an injective 
    order-preserving map $m \colon \NC(Q) \to \NC(P)$. Thus, $\NC(Q)$ is
    isomorphic to a subposet of $\NC(P)$. 
    
    Moreover, if $M$ is a maximal 
    element of $\HC(n)$, then it is represented by a configuration of $n$ 
    points in convex position, and thus $\NC(M)$ is isomorphic to the 
    classical noncrossing partition lattice $\NC(n)$. It then follows that
    for all $P \in \HC(n)$, we have that $\NC(P)$ is isomorphic to a
    subposet of $\NC(n)$.
\end{proof}

The poset $\NC(n)$ has a wealth of interesting properties; to
name a few, it is \emph{bounded} (it has a unique minimum $\hat{0}$
and a unique maximum $\hat{1}$), it is a \emph{lattice} (for all $x,y 
\in \NC(n)$, there is a unique meet $x \wedge y$ and a unique join 
$x \vee y$), and it is \emph{graded} with rank function 
$\rk \colon \NC(P) \to \mathbb{Z}$ given by defining $\rk(\pi)$ to
be $n-k$, where $k$ is the number of blocks in $\pi$. 

Some of the properties satisfied by $\NC(n)$ are also satisfied by the
more general definition of $\NC(P)$. For example, $\NC(P)$ is a bounded
lattice for every choice of convexity class $P$ \cite[Proposition~2.3]{cdhm24}.
In addition, it was shown in \cite[Theorem~2.3]{dfjllpr25} that 
when $P$ is a hull configuration, $\NC(P)$ is graded. In the remainder 
of this section, we show that in certain circumstances, $\NC(P)$ has a 
symmetric chain decomposition, a property which is also shared by $\NC(n)$.

\begin{defn}
    Let $P \in \HC(n)$. A subset of $\NC(P)$ is \emph{centered}
    if for each $k$, it has an equal number of elements at rank $k$ 
    and at rank $n-k$. A totally ordered subset of a poset
    is called a \emph{chain}, and a chain $\pi_1 < \cdots < \pi_k$ 
    is \emph{saturated} if each $\pi_i < \pi_{i+1}$ is a covering relation.
    Finally, $\NC(P)$ has a \emph{symmetric chain decomposition} 
    if it can be expressed as the disjoint union of centered saturated chains.
\end{defn}

Our approach for Theorem~\ref{mainthm:scd} is similar to those of
Theorem~3.2 and Theorem~4.2 in \cite{dfjllpr25}. In particular,
we use an inductive argument to express $\NC(P)$ as the disjoint 
union of centered subposets, each of which has a symmetric chain
decomposition. To this end, we introduce some notation for special
elements and intervals in $\NC(P)$ when $P$ has a blank side.

\begin{figure}
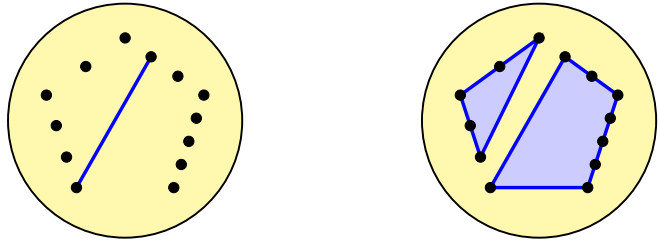

    \centering
    \includestandalone[width=0.7\textwidth]{fig-alpha-beta}
    \caption{Elements $\alpha_{3,2}$ (left) and $\beta_{3,2}$ (right) for the configuration
    $P$ depicted in Figure~\ref{fig:conf-shape}, as described in Definition~\ref{def:alpha-beta-elements}.}
    \label{fig:alpha-beta-elements}
\end{figure}

\begin{defn}
    \label{def:alpha-beta-elements}
    Suppose $P\in\HC(n)$ has at least one blank side; without loss of generality, let
    $\shape(P) = [0; c_2; \cdots; c_k]$. For each choice of $i$ and $j$ such that $2\leq i<k$ and $j\leq c_i$, we define two elements $\alpha_{ij}$ and $\beta_{ij}$ of $\NC(P)$, as well as two subsets $A_{ij}$ and $B_{ij}$ of $P$. First, let 
    $\alpha_{ij}$ be the unique atom in $\NC(P)$ with non-singleton block 
    $\{z_{1,0},z_{i,j}\}$, and let $\beta_{ij}$ be the unique coatom in $\NC(P)$ such that one block contains all points from $z_{1,0}$ to $z_{i,j}$ in the counterclockwise
    order around the boundary of $\conv(P)$, and another (nonempty) block contains all other elements of $P$. See Figure~\ref{fig:alpha-beta-elements} for an illustration and note that $\alpha_{ij}\leq\beta_{ij}$. Also, define $A_{ij}$ to be
    the subset of $P$ consisting of all points in the counterclockwise order from
    $z_{2,0}$ to $z_{i,j}$, and define $B_{ij} = P - (A_{ij} \cup \{z_{1,0}\})$;
    see Figure~\ref{fig:subconfigurations}.
\end{defn}

\begin{figure}
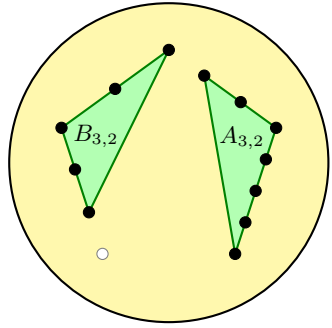

    \centering
    \includestandalone{fig-A-B}
    \caption{Subsets $A_{3,2}$ and $B_{3,2}$ of the configuration
    $P$ depicted in Figure~\ref{fig:conf-shape}, as described in Definition~\ref{def:alpha-beta-elements}. Note in particular that $A_{3,2} \cup B_{3,2} = P - \{z_{1,0}\}$.}
    \label{fig:subconfigurations}
\end{figure}

\begin{lem}
\label{lem:interval-product}
    Suppose $P \in \HC(n)$ has at least one blank side, and define the elements
    $\alpha_{ij}$ and $\beta_{ij}$ and the subsets $A_{ij}$ and $B_{ij}$ as
    in Definition~\ref{def:alpha-beta-elements}. Then 
    the interval $[\alpha_{ij},\beta_{ij}]$ is 
    isomorphic to the product $\NC(A_{ij})\times\NC(B_{ij})$.   
\end{lem}

\begin{proof}
    It is straightforward to see that $[\alpha_{ij},\beta_{ij}]$ is isomorphic to the product of $\NC(B_{ij})$ and the subposet of $\NC(A_{ij}\cup\{z_{1,0}\})$ consisting of all partitions where $z_{1,0}$ and $z_{i,j}$ share a block. We know from  \cite[Lemma~2.6]{dfjllpr25} that this subposet of $\NC(A_{ij}\cup\{z_{1,0}\})$ is isomorphic to $\NC(A_{ij})$, so it follows that $[\alpha_{ij},\beta_{ij}]$ is isomorphic to $\NC(A_{ij})\times\NC(B_{ij})$.
\end{proof}

\begin{lem}
    \label{lem:disjoint-union}
    Suppose $P\in\HC(n)$ with $\shape(P) = [0; c_2; \cdots; c_k]$. Let $X$ be the subposet of $\NC(P)$ consisting of all elements in which $z_{1,0}$ is either a singleton 
    or in a block containing $z_{k,c_k}$. Then 
    $X \cong \NC(P - \{z_{1,0}\}) \times \bool(1)$, and $\NC(P)$ is the disjoint union
    of $X$ and the intervals of the form $[\alpha_{ij},\beta_{ij}]$ where 
    $2 \leq i \leq k-1$ and $0 \leq j \leq c_i$.
\end{lem}

\begin{proof}
    For the first claim, note that the subposet of $\NC(P)$ in which $z_{1,0}$ is a singleton is isomorphic to $\NC(P-\{z_{1,0}\})$, as is the subposet of $\NC(P)$ in which $z_{1,0}$ shares a block with $z_{k,c_k}$. Each element in the former subposet is
    covered by a unique element of the latter, obtained by combining the block containing
    $z_{k,c_k}$ with the singleton block $\{z_{1,0}\}$. From here, it is straightforward to see that $X = \NC(P-\{z_{1,0}\})\times\bool(1)$.

    Next, let $\pi \in \NC(P)$ and note that there are three distinct 
    possibilities for $z_{1,0}$ in $\pi$. If $\{z_{1,0}\}$ is a singleton block of
    $\pi$, then $\pi \in X$. If $z_{1,0}$ shares a block with $z_{k,c_k}$, then 
    we again have that $\pi \in X$. Otherwise, let $z_{i,j}$ be the last point in
    $P$ which shares a block with $z_{1,0}$ with respect to the counterclockwise
    order on $P$ which begins at $z_{1,0}$. Necessarily we have 
    $2 \leq i \leq k-1$ and $0 \leq j \leq c_i$, and in particular $\pi$ belongs to
    $[\alpha_{ij},\beta_{ij}]$ and no other interval of this type. Therefore, 
    $\NC(P)$ can be expressed as the desired disjoint union.
\end{proof}

We are now ready to prove Theorem~\ref{mainthm:scd}.

\begin{thm}
    \label{thm:scd}
    If $P \in \HC(n)$ has at least one blank side, then $\NC(P)$ has
    a symmetric chain decomposition.
\end{thm}

\begin{proof}
    Without loss of generality, suppose that $\shape(P)=[0,c_2,\hdots,c_k]$.
    We will proceed by induction on $n$, the number of points in $P$.
    First, consider the base case $n=3$. Then the elements of $P$ lie either
    on a common line segment, in which case $\NC(P) \cong \bool(2)$, or 
    on the vertices of a triangle, in which case $\NC(P) \cong \NC(3)$.
    In both cases, $\NC(P)$ has a symmetric chain decomposition by inspection.

    Now suppose that the theorem holds for all configurations with at most
    $n-1$ points and at least one blank side. By Lemmas~\ref{lem:interval-product} and \ref{lem:disjoint-union}, we know that $\NC(P)$ is the disjoint union of 
    $X$, which is isomorphic to $\NC(P-\{z_{1,0}\}) \times \bool(1)$, 
    and the intervals of the form $[\alpha_{ij},\beta_{ij}]$, each of which
    is isomorphic to $\NC(A_{ij}) \times \NC(B_{ij})$. Since $P - \{z_{1,0}\}$,
    $A_{ij}$, and $B_{ij}$ are configurations with at least one blank side 
    and strictly fewer points than $P$, we know that their associated lattices
    of noncrossing partitions admit symmetric chain decompositions by 
    the inductive hypothesis. Since the property of having a symmetric 
    chain decomposition is preserved under direct products of posets, we
    can therefore say that $\NC(P)$ is the disjoint union of centered 
    subposets, each of which admits a symmetric chain decomposition,
    and it follows that $\NC(P)$ has one as well. By induction, the
    claim holds for all hull configurations $P$ with at least one
    blank side.
\end{proof}

\section{Maximal Boolean subposets}
\label{sec:max-boolean}

In this section, we show that a correspondence between $\NC(n)$ and 
noncrossing trees can be generalized to $\NC(P)$, and then
leverage this to prove Theorem~\ref{mainthm:boolean}.

\begin{defn}
    Let $P \in \CC(n)$. A \emph{noncrossing forest} on
    $P$ is an embedded graph with no cycles in $\mathbb{C}$ with vertex set $P$
    and geodesic edges. Note that a path between two 
    configurations in the same convexity class induces an isomorphism 
    of embedded graphs on the two vertex sets, so we are able to 
    work with $P$ as a convexity class rather than a specific 
    configuration. 
    If $\mu_1$ and $\mu_2$ are
    noncrossing forests with the same vertex set 
    such that $\mu_1$ is a subgraph of $\mu_2$,
    then we write $\mu_1 \subseteq \mu_2$. A \emph{noncrossing tree}
    is a connected noncrossing forest. Finally, if $\mu$ is a noncrossing
    forest on the vertex set $P$, define $\partition(\mu)$ to be the
    partition of $P$ in which distinct elements $x$ and $y$ share a 
    block whenever they belong to the same connected component of $\mu$.
\end{defn}

If $M$ is a maximal element of $\HC(n)$, then $\NC(M) \cong \NC(n)$ 
and the noncrossing trees on $M$ are enumerated by the 
Fuss--Catalan number $C_{n-1}^{(3)} = \frac{1}{2n-1}\binom{3n-3}{n-1}$.
Moreover, it was previously shown that there is a bijection between
these trees and the subposets of $\NC(n)$ which are isomorphic to 
$\bool(n-1)$, the Boolean lattice of height $n-1$.

\begin{thm}[\protect{\cite[Proposition~4.4]{hks16} and \cite[Proposition~5.6]{hypertrees}}]
    \label{thm:classical-boolean-bijection}
    Let $M$ be a maximal element of $\HC(n)$. 
    Then there is a one-to-one correspondence between
    the maximal Boolean subposets of $\NC(M)$ and 
    the noncrossing trees on $M$.
\end{thm}

The correspondence above can be extended to the
more general case of $\NC(P)$, but we first need to define a special 
type of noncrossing tree.

\begin{defn}
    \label{def:convex-geodesics}
    Let $P \in \HC(n)$. We say that a noncrossing forest $\mu$ on 
    $P$ has \emph{convex geodesics} if for each 
    subtree $\tau \subseteq \mu$, the only elements of $P$ in
    $\conv(\tau)$ are the vertices of $\tau$. Equivalently, for each
    pair of distinct vertices $x,y \in P$ which belong to the same
    connected component of $\mu$, the unique path from 
    $x$ to $y$ in $\mu$ has a convex hull which includes only the elements of 
    $P$ which lie along the path.
\end{defn}

While we will not need it in this article, it can also be shown that a
noncrossing tree $\tau$ on $P$ has convex geodesics if and only all
leaves of $\tau$ (i.e. the vertices of degree 1) lie on corners of
$\conv(P)$.

\begin{lem}
    \label{lem:convex-geodesics-implies-noncrossing}
    If $\mu$ is a noncrossing forest on the vertex set $P$
    with convex geodesics, then $\partition(\mu)$ is a noncrossing
    partition of $P$.
\end{lem}

\begin{proof}
    Each pair of distinct blocks $A_1$ and $A_2$ in $\partition(\mu)$
    can be viewed as the vertex sets for distinct subtrees $\tau_1$ and 
    $\tau_2$ in $\mu$. Since $\mu$ has convex geodesics, we know that
    $\conv(A_i) \cap P = \conv(\tau_i) \cap P = A_i$, so the convex hull
    of each $A_i$ does not contain any additional points of $P$. 
    Moreover, the convex hulls of $A_1$ and $A_2$ do not have intersecting
    interiors since $\tau_1$ and $\tau_2$ do not overlap because
    $\mu$ is noncrossing. Therefore, $\conv(A_1)$ and $\conv(A_2)$ are
    disjoint, so $\partition(\mu)$ is noncrossing.
\end{proof}

\begin{lem}
    \label{lem:rank-convexity}
    Let $\pi_1,\ldots,\pi_k$ be atoms in $\NC(P)$ such that the corresponding
    edges $e_1,\ldots,e_k$ form a tree with vertex set $Q \subseteq P$. 
    Then $\rk(\pi_1\vee \cdots \vee \pi_k) = k$ if and only if 
    $\conv(e_1,\ldots,e_k) \cap P = Q$. 
\end{lem}

\begin{proof}
    Let $\pi \in \NC(P)$ denote the join $\pi_1\vee \cdots \vee \pi_k$
    and note that since $e_1,\ldots,e_k$ form a connected graph,
    $\pi$ is the partition of $P$ which has 
    $\conv(e_1,\ldots,e_k) \cap P$ as its unique non-singleton block. 
    Since $Q\subseteq \conv(e_1,\ldots,e_k) \cap P$
    and $|Q| = k+1$ (because $e_1,\ldots,e_k$ form a tree),
    we see that $\rk(\pi) \leq k$, with equality precisely when 
    $\conv(e_1,\ldots,e_k) \cap P = Q$.
\end{proof}

We now state and prove a generalization of 
Theorem~\ref{thm:classical-boolean-bijection}, which forms the 
first part of Theorem~\ref{mainthm:boolean}.

\begin{thm}
    \label{thm:boolean-bijection}
    Let $P\in \HC(n)$. Then there is a 
    one-to-one correspondence between the maximal Boolean 
    subposets of $\NC(P)$ and the noncrossing trees on $P$ with convex geodesics.
\end{thm}

\begin{proof}
    First, let $\tau$ be a noncrossing tree on $P$ with convex geodesics
    and recall that each subforest $\mu \subseteq \tau$ determines a noncrossing
    partition $\partition(\mu)$ of $P$ by 
    Lemma~\ref{lem:convex-geodesics-implies-noncrossing}.
    It is straightforward
    to see that the $2^{n-1}$ subforests of $\tau$ produce distinct
    partitions, and $\partition(\mu_1) \subseteq \partition(\mu_2)$ in $\NC(P)$ 
    if and only if $\mu_1 \subseteq \mu_2$, so it follows that each noncrossing
    tree $\tau$ with convex geodesics on $P$ determines a maximal Boolean subposet 
    (i.e. a copy of $\bool(n-1)$) in $\NC(P)$---call this subposet 
    $\bool(\tau)$. Since the atoms of $\bool(\tau)$ correspond to the edges of
    $\tau$, we can then see that the map $\tau \mapsto \bool(\tau)$ is
    an injection from the set of noncrossing trees on $P$ with convex geodesics
    to the set of maximal Boolean subposets of $\NC(P)$.

    Now, let $\mathcal{B}$ be a maximal Boolean subposet of $\NC(P)$, i.e. 
    $\mathcal{B} \cong \bool(n-1)$. Then the maximum element of $\NC(P)$,
    which consists of the single block $P$, is the join of the $n-1$ atoms of
    $\mathcal{B}$, and it follows that the edges corresponding to these atoms
    form a connected graph with vertex set $P$. Since there are $n-1$ edges
    and $|P| = n$, we know by the Euler characteristic that this connected 
    graph is a tree, and since the join of any two atoms has rank 2 in 
    $\bool(n-1)$, this tree must be noncrossing---call it $\tau$. 
    In a Boolean lattice, the join of any $k$ atoms has rank $k$, so 
    by Lemma~\ref{lem:rank-convexity}, every subtree of $\tau$
    has a convex hull which intersects $P$ in its vertex set. In other words,
    $\tau$ has convex geodesics. It follows that $\mathcal{B} = \bool(\tau)$
    and the proof is complete.
\end{proof}

In addition to Theorem~\ref{thm:classical-boolean-bijection}, it
was shown in \cite{hks16} and \cite{hypertrees} that $\NC(n)$ is
the union of its maximal Boolean subposets. To close this section,
we prove the analogous result for $\NC(P)$ when $P$ is a hull configuration.
First, we will need a way of perturbing noncrossing trees by ``sliding''
edges.

\begin{defn}
    \label{def:slide}
    Let $P \in \HC(n)$ and let $\tau$ be a noncrossing tree on $P$ with convex geodesics. 
    Suppose that $p$, $q$ and $r$
    are vertices of $\tau$ and $e = \{p,q\}$, $f = \{q,r\}$
    are edges of $\tau$ which are adjacent in the cyclic order
    of edges incident to $q$. If we replace $e$ with a new edge $e' = \{p,r\}$,
    then the resulting graph $\tau'$ is still a tree on $P$, which 
    we say was obtained by \emph{sliding $e$ along $f$}. See
    Figure~\ref{fig:tree-slide} for an illustration. 
\end{defn}

\begin{lem}
    \label{lem:slide}
    Let $P \in \HC(n)$, let $\tau$ be a noncrossing tree on $P$ with convex geodesics,
    and suppose that $\tau'$ is another noncrossing tree on $P$ with convex geodesics
    which is the result of sliding the edge $e = \{p,q\}$ along $f = \{q,r\}$. 
    If $\pi \in \bool(\tau)$, then $\pi$ is also an element of $\bool(\tau')$ 
    if and only if the block containing $p$ either contains both $q$ and $r$
    or contains neither $q$ nor $r$.
\end{lem}

\begin{figure}
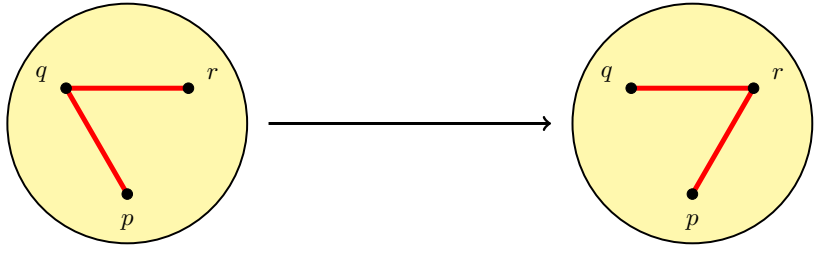

    \centering
    \includestandalone[width=0.9\textwidth]{fig-tree-slide}
    \caption{On the left, a tree on the vertex set $\{p,q,r\}$;
    on the right, the result of sliding the edge $\{p,q\}$
    along the edge $\{q,r\}$.}
    \label{fig:tree-slide}
\end{figure}

\begin{proof}
    Suppose $\pi \in \bool(\tau)$ and let $A$ be the set of atoms 
    in $\bool(\tau)$ below $\pi$, i.e. the join of all elements in $A$
    is $\bigvee A = \pi$. Also, let $\pi_{pq}$ and $\pi_{pr}$ be the atoms 
    of $\NC(P)$ with unique non-singleton block $\{p,q\}$ and $\{p,r\}$, respectively. 
    If $\pi_{pq} \not\in A$, then $\pi$ remains
    in $\bool(\tau')$ since all elements of $A$ remain in the new 
    poset. Note that this corresponds to the case where $p$ belongs to a distinct
    block from the block (or blocks) containing $q$ and $r$, since 
    $\pi_{pq} \not\leq \pi$ and $\pi_{pr} \not\leq \pi$. If on the other hand
    $\pi_{pq} \in A$ and $\pi \in \bool(\tau')$, then it must be that $\pi_{pr} \leq \pi$, 
    else $\pi$ could be represented as the join of elements in $A-\{\pi_{pq}\}$, 
    which contradicts our assumption. Conversely, if $\pi_{pq} \leq \pi$
    and $\pi_{pr} \leq \pi$, then the join of all elements in $(A-\{\pi_{pq}\})\cup \{\pi_{pr}\}$
    must be equal to $\pi$, meaning that $\pi \in \bool(\tau')$. Since this 
    corresponds to the case where $p$, $q$ and $r$ share a block in $\pi$, 
    the proof is complete.
\end{proof}

We are now ready to complete the proof of Theorem~\ref{mainthm:boolean}.

\begin{thm}
    \label{thm:boolean-union}
    For each $P\in \HC(n)$, the lattice $\NC(P)$ is a union of 
    maximal Boolean subposets.
\end{thm}

\begin{figure}
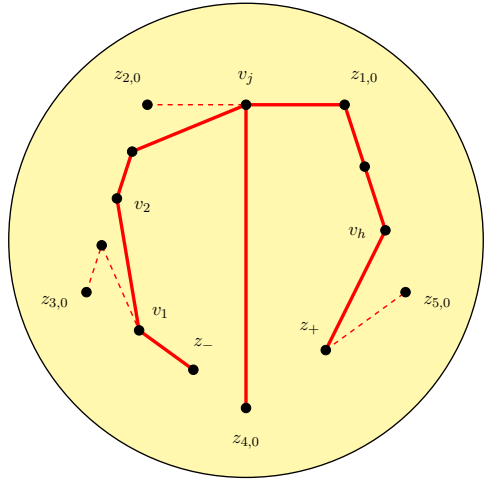

    \centering
    \includestandalone[width=0.5\textwidth]{fig-thm-boolean}
    \caption{A noncrossing tree $\tau$ on a hull configuration $Q$, with the
    geodesic between vertices $z_-$ and $z_+$ highlighted.}
    \label{fig:geodesic}
\end{figure}

\begin{proof}
    From \cite{hks16} and \cite{hypertrees}, we already know that 
    $\NC(P)$ is a union of maximal Boolean subposets when $P$ is
    a maximal element of $\HC(n)$. Now, suppose that there is a hull 
    configuration $Q \in \HC(n)$ such that $\NC(Q)$ is a union
    of maximal Boolean subposets, and let $P = m(Q)$ be the configuration
    obtained by applying the elementary collapse $m$ to $Q$.
    Since every element of $\HC(n)$ can be obtained
    by applying a finite sequence of elementary collapses to a maximal
    element, it suffices to show that $\NC(P)$ is a union of maximal 
    Boolean subposets. 

    Let $\pi \in \NC(P)$ with $\shape(\pi) = [c_1;\cdots;c_k]$ and suppose 
    that $z_{i,0}$ is the external vertex of $Q$ which is made into an 
    internal collinearity of $P$ under the elementary collapse $m$. Then 
    $\pi = m(\rho)$ for some $\rho \in \NC(Q)$, and by combining 
    Theorem~\ref{thm:boolean-bijection} with the assumption that $\NC(Q)$ is 
    a union of maximal Boolean subposets, we know that $\rho \in \bool(\tau)$ 
    for some noncrossing tree $\tau$ on $Q$ 
    with convex geodesics. If the image $m(\tau)$ is a noncrossing tree on $P$
    with convex geodesics, then $\pi$ belongs to the maximal Boolean subposet
    $\bool(m(\tau))$ and we are done. 
    
    Suppose instead that $m(\tau)$ either fails to have convex geodesics or
    fails to be noncrossing. As a convenient shorthand, let 
    $z_-$ and $z_+$ be the elements of $Q$ which appear immediately before and 
    immediately after $z_{i,0}$ in the counterclockwise cyclic order on $Q$.
    Concretely, $z_- = z_{i-1,c_{i-1}}$ and $z_+$ is either
    $z_{i,1}$ if $c_i > 0$ or $z_{i+1,0}$ if $c_i = 0$. If $m(\tau)$ fails
    to have convex geodesics, then this means by definition that 
    $z_{i,0}$ does not lie on the geodesic from $z_{i-1,0}$ to $z_{i+1,0}$,
    which in particular means that $z_{i,0}$ does not lie on the geodesic in 
    $\tau$ from $z_-$ to $z_+$. If $m(\tau)$ fails to be noncrossing,
    then $\tau$ includes an edge between the sides labeled $i-1$ and $i$
    in $\conv(Q)$ which does not include $z_{i,0}$, and we are again in the 
    preceding case: $z_{i,0}$ does not lie on the geodesic in $\tau$ from $z_-$ to $z_+$.
    
    Let us relabel the vertices in the geodesic in $\tau$ from $z_-$ to $z_+$
    in the following manner: $z_- = v_0, v_1, \ldots, v_h, v_{h+1} = z_+$.
    Then $z_{i,0}$ is the endpoint of a unique edge in $\tau$, and the other endpoint
    is some $v_j$ on the aforementioned geodesic---see Figure~\ref{fig:geodesic}
    for an illustration. If $v_j$ is in the same block of $\rho$ as $z_{i,0}$, then 
    by Lemma~\ref{lem:slide},
    we can slide the edge connecting $v_j$ to $v_{j-1}$ (or the edge connecting $v_j$ to
    $v_{j+1}$) along the edge connecting
    $v_j$ to $z_{i,0}$ to obtain a noncrossing tree $\tau'$ with convex geodesics
    on $Q$ such that $\rho \in \bool(\tau')$ and $m(\tau')$ is a noncrossing tree
    with convex geodesics on $P$. 

    Now suppose that $v_j$ and $z_{i,0}$ do not share a block in $\rho$,
    and observe that this implies that $z_{i,0}$ must be a singleton
    in $\rho$, else $\tau$ contains an edge between two points in the 
    block containing $z_{i,0}$, and this necessarily crosses an edge
    in the geodesic in $\tau$ from $z_-$ to $z_+$, contradicting the
    assumption that $\tau$ is noncrossing. Since we know that $\pi = m(\rho)$
    is a noncrossing partition of the configuration $P = m(Q)$, it can't be the
    case that a point on side $i-1$ of $\conv(P)$ shares a block with a
    point on side $i$ while $z_{i,0}$ is a singleton block.
    It then follows that there is some $v_l$ of minimal distance
    from $v_j$ such that $v_l$ and $v_j$ are not in the same block, but 
    every point on the geodesic between those two is in the same block as $v_j$.
    Then we can slide the edge between $v_j$ and $z_{i,0}$ along the geodesic 
    until we reach the vertex before $v_l$, which must not lie on side $i-1$ or side $i$, at which point we are in the setting of the previous case, and the proof is complete.
\end{proof}

\begin{figure}
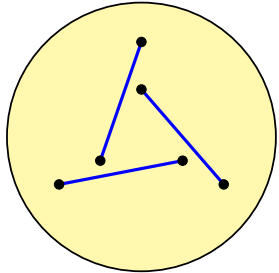

    \centering
    \includestandalone[width = 0.3\textwidth]{fig-boolean-nonexample}
    \caption{This noncrossing partition with vertex set $P$ 
    cannot be expressed as $\partition(\mu)$ where $\mu$ is a subforest 
    of a noncrossing tree with convex geodesics on $P$, so it does not
    belong to any maximal Boolean subposet of $\NC(P)$.}
    \label{fig:boolean-union-nonexample}
\end{figure}

It is worth noting that there are some configurations $P$ (which are not
hull configurations) such that 
$\NC(P)$ is not a union of maximal Boolean
subposets. See Figure~\ref{fig:boolean-union-nonexample} for
an example and observe that in this instance, $\NC(P)$ is not even
graded \cite[Section~4]{cdhm24}.

\section*{Acknowledgments}

We are grateful to Lafayette College for support in Spring 2025, when
much of the research for this article took place. The first author is grateful
to Jon McCammond for helpful conversations. The first author is
partially supported by NSF grant DMS-2532608.

\bibliographystyle{amsalpha}
\bibliography{nc-hull}

\end{document}

%% file: tikz-styles.tex
\tikzstyle{PurpleLine}=[line width=0.3mm,color=blue-violet,text=black]
\tikzstyle{PurplePoly}=[PurpleLine,fill=blue-violet!30]
\tikzstyle{BlueLine}=[line width=0.3mm,color=blue,text=black]
\tikzstyle{BluePoly}=[BlueLine,fill=blue!20]
\tikzstyle{RedLine}=[line width=0.3mm,color=red,text=black]
\tikzstyle{RedPoly}=[RedLine,fill=red!20]
\tikzstyle{GreenLine}=[thick,draw=black!30!green,text=black]
\tikzstyle{GreenPoly}=[thick,draw=green!50!black,fill=green!30,join=bevel]
\tikzstyle{OrangeLine}=[thick,color=orange]
\tikzstyle{GrayLine}=[thick,color=black!50!gray]
\tikzstyle{GrayPoly}=[GrayLine,fill=gray!20]
\tikzstyle{dot}=[shape=circle,draw,color=black,fill=black,inner sep=1.5pt]
\tikzstyle{bigdot}=[dot,inner sep=2.4pt]
\tikzstyle{littledot}=[dot,inner sep=1.2pt]
\tikzstyle{disk}=[thick,shape=circle,draw,color=black,fill=yellow!10]
\tikzstyle{plate}=[thick,shape=rectangle,draw,color=black,fill=yellow!10,
	rounded corners,minimum size=1.1cm]
\tikzstyle{dot}=[shape=circle,draw,color=black,fill=black,inner sep=1.5pt]
\tikzstyle{opendot}=[dot,fill=white]

\tikzstyle{YellowRect} = [shape=rectangle,rounded corners,draw,fill=yellow!40,minimum width = 2cm, minimum height=1cm]

%% file: figures/fig-elem-collapse.tex
\begin{tikzpicture}

    \begin{scope}[shift={(-8,0)}, rotate=180]
        \node [shape=rectangle,draw,fill=green!20, inner sep = 5cm] () at (0,0) {};
    
        \foreach \i in {1,...,5} {
            \coordinate (a\i) at (72*\i+72+90:3.5);
        }
    
        \coordinate (b1) at ($(a4)!0.5!(a5)$);
        \coordinate (b2) at ($(a5)!0.33!(a1)$);
        \coordinate (b3) at ($(a5)!0.67!(a1)$);
        \coordinate (b4) at ($(a2)!0.25!(a3)$);
        \coordinate (b5) at ($(a2)!0.5!(a3)$);
        \coordinate (b6) at ($(a2)!0.75!(a3)$);
        \coordinate (b7) at ($(a3)!0.33!(a4)$);
        \coordinate (b8) at ($(a3)!0.67!(a4)$);
        \coordinate (b9) at ($(a1)!0.5!(a2)$);
    
        \foreach \i in {1,...,5} {
            \draw (a\i) node [dot] {};
        }
    
        \foreach \i in {1,...,8,9} {
            \draw (b\i) node [dot] {};
        }
    \end{scope}

    \draw[->, line width = 1.8] (-2,0) -- (2,0);

    \begin{scope}[shift={(8,0)}, rotate=180]
        \node [shape=rectangle,draw,fill=green!20, inner sep = 5cm] () at (0,0) {};
    
        \foreach \i in {1,2,3,5} {
            \coordinate (a\i) at (72*\i+72+90:3.5);
        }

        \coordinate (a4) at ($(a3)!0.5!(a5)$);
    
        \coordinate (b1) at ($(a4)!0.5!(a5)$);
        \coordinate (b2) at ($(a5)!0.33!(a1)$);
        \coordinate (b3) at ($(a5)!0.67!(a1)$);
        \coordinate (b4) at ($(a2)!0.25!(a3)$);
        \coordinate (b5) at ($(a2)!0.5!(a3)$);
        \coordinate (b6) at ($(a2)!0.75!(a3)$);
        \coordinate (b7) at ($(a3)!0.33!(a4)$);
        \coordinate (b8) at ($(a3)!0.67!(a4)$);
        \coordinate (b9) at ($(a1)!0.5!(a2)$);
    
        \foreach \i in {1,...,5} {
            \draw (a\i) node [dot] {};
        }
    
        \foreach \i in {1,...,8,9} {
            \draw (b\i) node [dot] {};
        }
    \end{scope}

\end{tikzpicture}

%% file: figures/fig-configuration-example.tex
\begin{tikzpicture}

    \node [shape=rectangle,draw,fill=green!20, inner sep = 5cm] () at (0,0) {};
    \tikzstyle{dot}=[shape=circle,draw,color=black,fill=black,inner sep=2.25pt]

    \foreach \i in {1,...,5} {
        \coordinate (a\i) at (72*\i+72+90:3.5);
    }

    \coordinate (b1) at ($(a4)!0.5!(a5)$);
    \coordinate (b2) at ($(a5)!0.33!(a1)$);
    \coordinate (b3) at ($(a5)!0.67!(a1)$);
    \coordinate (b4) at ($(a2)!0.25!(a3)$);
    \coordinate (b5) at ($(a2)!0.5!(a3)$);
    \coordinate (b6) at ($(a2)!0.75!(a3)$);
    \coordinate (b7) at ($(a3)!0.33!(a4)$);
    \coordinate (b8) at ($(a3)!0.67!(a4)$);

    \foreach \i in {1,...,5} {
        \draw (a\i) node [dot] {};
    }

    \foreach \i in {1,...,8} {
        \draw (b\i) node [dot] {};
    }

    \foreach \i in {1,...,5} {
        \draw ($(0,0)!1.2!(a\i)$) node () {\large $z_{\i,0}$};
    }

    \draw ($(0,0)!1.2!(b1)$) node () {\large $z_{4,1}$};
    \draw ($(0,0)!1.2!(b2)$) node () {\large $z_{5,1}$};
    \draw ($(0,0)!1.2!(b3)$) node () {\large $z_{5,2}$};
    \draw ($(0,0)!1.2!(b4)$) node () {\large $z_{2,1}$};
    \draw ($(0,0)!1.2!(b5)$) node () {\large $z_{2,2}$};
    \draw ($(0,0)!1.2!(b6)$) node () {\large $z_{2,3}$};
    \draw ($(0,0)!1.2!(b7)$) node () {\large $z_{3,1}$};
    \draw ($(0,0)!1.2!(b8)$) node () {\large $z_{3,2}$};
\end{tikzpicture}

%% file: figures/fig-cc-lower-set.tex
\begin{tikzpicture}
    \newcommand{\x}{4}
    \newcommand{\y}{8}


    \begin{scope}[shift={(0,2*\y)}]
        \node [shape=rectangle,draw,fill=green!20, inner sep = 1.5cm] (top) at (0,0) {};
        \node[bigdot,fill=red] () at (45:1) {};
        \node[bigdot,fill=yellow!90!black] () at (135:1) {};
        \node[bigdot,fill=green!80!black] () at (225:1) {};
        \node[bigdot,fill=blue] () at (315:1) {};

        \node () at (45:1.4) {1};
        \node () at (135:1.4) {2};
        \node () at (225:1.4) {3};
        \node () at (315:1.4) {4};
    \end{scope}


    \begin{scope}[shift={(-1.5*\x,\y)}]
        \node [shape=rectangle,draw,fill=green!20, inner sep = 1.5cm] (m1) at (0,0) {};
        \node[bigdot,fill=red] (d) at (90:0.7) {};
        \node[bigdot,fill=yellow!90!black] (a) at (210:0.7) {};
        \node[bigdot,fill=blue] (b) at (330:0.7) {};
        \node[bigdot,fill=green!80!black] (c) at ($(a)!0.5!(b)$) {};

        \node () at ($(d)+(0,0.5)$) {1};
        \node () at ($(a)+(0,-0.5)$) {2};
        \node () at ($(c)+(0,-0.5)$) {3};
        \node () at ($(b)+(0,-0.5)$) {4};
    \end{scope}

    \begin{scope}[shift={(1.5*\x,\y)}]
        \node [shape=rectangle,draw,fill=green!20, inner sep = 1.5cm] (m2) at (0,0) {};
        \node[bigdot,fill=yellow!90!black] (d) at (90:0.7) {};
        \node[bigdot,fill=green!80!black] (a) at (210:0.7) {};
        \node[bigdot,fill=red] (b) at (330:0.7) {};
        \node[bigdot,fill=blue] (c) at ($(a)!0.5!(b)$) {};

        \node () at ($(d)+(0,0.5)$) {2};
        \node () at ($(a)+(0,-0.5)$) {3};
        \node () at ($(c)+(0,-0.5)$) {4};
        \node () at ($(b)+(0,-0.5)$) {1};
    \end{scope}
    
    \begin{scope}[shift={(0.5*\x,\y)}]
        \node [shape=rectangle,draw,fill=green!20, inner sep = 1.5cm] (m3) at (0,0) {};
        \node[bigdot,fill=green!80!black] (d) at (90:0.7) {};
        \node[bigdot,fill=blue] (a) at (210:0.7) {};
        \node[bigdot,fill=yellow!90!black] (b) at (330:0.7) {};
        \node[bigdot,fill=red] (c) at ($(a)!0.5!(b)$) {};

        \node () at ($(d)+(0,0.5)$) {3};
        \node () at ($(a)+(0,-0.5)$) {4};
        \node () at ($(c)+(0,-0.5)$) {1};
        \node () at ($(b)+(0,-0.5)$) {2};
    \end{scope}
    
    \begin{scope}[shift={(-0.5*\x,\y)}]
        \node [shape=rectangle,draw,fill=green!20, inner sep = 1.5cm] (m4) at (0,0) {};
        \node[bigdot,fill=blue] (d) at (90:0.7) {};
        \node[bigdot,fill=red] (a) at (210:0.7) {};
        \node[bigdot,fill=green!80!black] (b) at (330:0.7) {};
        \node[bigdot,fill=yellow!90!black] (c) at ($(a)!0.5!(b)$) {};

        \node () at ($(d)+(0,0.5)$) {4};
        \node () at ($(a)+(0,-0.5)$) {1};
        \node () at ($(c)+(0,-0.5)$) {2};
        \node () at ($(b)+(0,-0.5)$) {3};
    \end{scope}

    
    \begin{scope}[shift={(-3.5*\x,0)}]
        \node [shape=rectangle,draw,fill=green!20, inner sep = 1.5cm] (b1) at (0,0) {};
        \node[bigdot,fill=yellow!90!black] (a) at (-1.05,0) {};
        \node[bigdot,fill=green!80!black] (b) at (-0.35,0) {};
        \node[bigdot,fill=red] (c) at (0.35,0) {};
        \node[bigdot,fill=blue] (d) at (1.05,0) {};

        \node () at ($(a)+(0,-0.5)$) {2};
        \node () at ($(b)+(0,-0.5)$) {3};
        \node () at ($(c)+(0,-0.5)$) {1};
        \node () at ($(d)+(0,-0.5)$) {4};
    \end{scope}

    \begin{scope}[shift={(-2.5*\x,0)}]
        \node [shape=rectangle,draw,fill=green!20, inner sep = 1.5cm] (b2) at (0,0) {};
        \node[bigdot,fill=yellow!90!black] (a) at (-1.05,0) {};
        \node[bigdot,fill=red] (b) at (-0.35,0) {};
        \node[bigdot,fill=green!80!black] (c) at (0.35,0) {};
        \node[bigdot,fill=blue] (d) at (1.05,0) {};

        \node () at ($(a)+(0,-0.5)$) {2};
        \node () at ($(b)+(0,-0.5)$) {1};
        \node () at ($(c)+(0,-0.5)$) {3};
        \node () at ($(d)+(0,-0.5)$) {4};
    \end{scope}

    \begin{scope}[shift={(-1.5*\x,0)}]
        \node [shape=rectangle,draw,fill=green!20, inner sep = 1.5cm] (b3) at (0,0) {};
        \node[bigdot,fill=yellow!90!black] (a) at (-1.05,0) {};
        \node[bigdot,fill=green!80!black] (b) at (-0.35,0) {};
        \node[bigdot,fill=blue] (c) at (0.35,0) {};
        \node[bigdot,fill=red] (d) at (1.05,0) {};

        \node () at ($(a)+(0,-0.5)$) {2};
        \node () at ($(b)+(0,-0.5)$) {3};
        \node () at ($(c)+(0,-0.5)$) {4};
        \node () at ($(d)+(0,-0.5)$) {1};
    \end{scope}

    \begin{scope}[shift={(-0.5*\x,0)}]
        \node [shape=rectangle,draw,fill=green!20, inner sep = 1.5cm] (b4) at (0,0) {};
        \node[bigdot,fill=red] (a) at (-1.05,0) {};
        \node[bigdot,fill=yellow!90!black] (b) at (-0.35,0) {};
        \node[bigdot,fill=green!80!black] (c) at (0.35,0) {};
        \node[bigdot,fill=blue] (d) at (1.05,0) {};

        \node () at ($(a)+(0,-0.5)$) {1};
        \node () at ($(b)+(0,-0.5)$) {2};
        \node () at ($(c)+(0,-0.5)$) {3};
        \node () at ($(d)+(0,-0.5)$) {4};
    \end{scope}

    \begin{scope}[shift={(0.5*\x,0)}]
        \node [shape=rectangle,draw,fill=green!20, inner sep = 1.5cm] (b5) at (0,0) {};
        \node[bigdot,fill=blue] (a) at (-1.05,0) {};
        \node[bigdot,fill=red] (b) at (-0.35,0) {};
        \node[bigdot,fill=yellow!90!black] (c) at (0.35,0) {};
        \node[bigdot,fill=green!80!black] (d) at (1.05,0) {};

        \node () at ($(a)+(0,-0.5)$) {4};
        \node () at ($(b)+(0,-0.5)$) {1};
        \node () at ($(c)+(0,-0.5)$) {2};
        \node () at ($(d)+(0,-0.5)$) {3};
    \end{scope}

    \begin{scope}[shift={(1.5*\x,0)}]
        \node [shape=rectangle,draw,fill=green!20, inner sep = 1.5cm] (b6) at (0,0) {};
        \node[bigdot,fill=green!80!black] (a) at (-1.05,0) {};
        \node[bigdot,fill=blue] (b) at (-0.35,0) {};
        \node[bigdot,fill=red] (c) at (0.35,0) {};
        \node[bigdot,fill=yellow!90!black] (d) at (1.05,0) {};

        \node () at ($(a)+(0,-0.5)$) {3};
        \node () at ($(b)+(0,-0.5)$) {4};
        \node () at ($(c)+(0,-0.5)$) {1};
        \node () at ($(d)+(0,-0.5)$) {2};
    \end{scope}

    \begin{scope}[shift={(2.5*\x,0)}]
        \node [shape=rectangle,draw,fill=green!20, inner sep = 1.5cm] (b7) at (0,0) {};
        \node[bigdot,fill=green!80!black] (a) at (-1.05,0) {};
        \node[bigdot,fill=blue] (b) at (-0.35,0) {};
        \node[bigdot,fill=yellow!90!black] (c) at (0.35,0) {};
        \node[bigdot,fill=red] (d) at (1.05,0) {};

        \node () at ($(a)+(0,-0.5)$) {3};
        \node () at ($(b)+(0,-0.5)$) {4};
        \node () at ($(c)+(0,-0.5)$) {2};
        \node () at ($(d)+(0,-0.5)$) {1};
    \end{scope}

    \begin{scope}[shift={(3.5*\x,0)}]
        \node [shape=rectangle,draw,fill=green!20, inner sep = 1.5cm] (b8) at (0,0) {};
        \node[bigdot,fill=green!80!black] (a) at (-1.05,0) {};
        \node[bigdot,fill=yellow!90!black] (b) at (-0.35,0) {};
        \node[bigdot,fill=blue] (c) at (0.35,0) {};
        \node[bigdot,fill=red] (d) at (1.05,0) {};

        \node () at ($(a)+(0,-0.5)$) {3};
        \node () at ($(b)+(0,-0.5)$) {2};
        \node () at ($(c)+(0,-0.5)$) {4};
        \node () at ($(d)+(0,-0.5)$) {1};
    \end{scope}


    \begin{scope}[line width = 0.4mm]
        \path (top.south) 
            edge (m1.north) 
            edge (m2.north) 
            edge (m3.north) 
            edge (m4.north);
    
        \path (m1.south) 
            edge (b1.north) 
            edge (b2.north) 
            edge (b3.north) 
            edge (b4.north);
        \path (m2.south) 
            edge (b3.north) 
            edge (b6.north) 
            edge (b7.north) 
            edge (b8.north);
        \path (m3.south) 
            edge (b1.north) 
            edge (b5.north) 
            edge (b2.north) 
            edge (b6.north);
        \path (m4.south) 
            edge (b8.north) 
            edge (b4.north) 
            edge (b5.north) 
            edge (b7.north);
    \end{scope}






\end{tikzpicture}

%% file: figures/fig-cc-upper-set.tex
\begin{tikzpicture}
    \newcommand{\x}{4}
    \newcommand{\y}{8}
    

    \begin{scope}[shift={(-4.5*\x,2*\y)}]
        \node [shape=rectangle,draw,fill=green!20, inner sep = 1.5cm] (t1) at (0,0) {};
        \node[bigdot,fill=red] (a) at (90:0.8) {};
        \node[bigdot,fill=yellow!90!black] (b) at (210:0.8) {};
        \node[bigdot,fill=blue] (c) at (330:0.8) {};
        \node[bigdot,fill=green!80!black] (d) at (0,0) {};

        \node () at (90:1.2) {1};
        \node () at (210:1.2) {2};
        \node () at (330:1.2) {4};
        \node () at (270:0.4) {3};
    \end{scope}

    \begin{scope}[shift={(-3.5*\x,2*\y)}]
        \node [shape=rectangle,draw,fill=green!20, inner sep = 1.5cm] (t2) at (0,0) {};
        \node[bigdot,fill=yellow!90!black] (a) at (90:0.8) {};
        \node[bigdot,fill=green!80!black] (b) at (210:0.8) {};
        \node[bigdot,fill=red] (c) at (330:0.8) {};
        \node[bigdot,fill=blue] (d) at (0,0) {};

        \node () at (90:1.2) {2};
        \node () at (210:1.2) {3};
        \node () at (330:1.2) {1};
        \node () at (270:0.4) {4};
    \end{scope}

    \begin{scope}[shift={(-2.5*\x,2*\y)}]
        \node [shape=rectangle,draw,fill=green!20, inner sep = 1.5cm] (t3) at (0,0) {};
        \node[bigdot,fill=green!80!black] (a) at (90:0.8) {};
        \node[bigdot,fill=blue] (b) at (210:0.8) {};
        \node[bigdot,fill=yellow!90!black] (c) at (330:0.8) {};
        \node[bigdot,fill=red] (d) at (0,0) {};

        \node () at (90:1.2) {3};
        \node () at (210:1.2) {4};
        \node () at (330:1.2) {2};
        \node () at (270:0.4) {1};
    \end{scope}

    \begin{scope}[shift={(-1.5*\x,2*\y)}]
        \node [shape=rectangle,draw,fill=green!20, inner sep = 1.5cm] (t4) at (0,0) {};
        \node[bigdot,fill=blue] (a) at (90:0.8) {};
        \node[bigdot,fill=red] (b) at (210:0.8) {};
        \node[bigdot,fill=green!80!black] (c) at (330:0.8) {};
        \node[bigdot,fill=yellow!90!black] (d) at (0,0) {};

        \node () at (90:1.2) {4};
        \node () at (210:1.2) {1};
        \node () at (330:1.2) {3};
        \node () at (270:0.4) {2};
    \end{scope}

    \begin{scope}[shift={(-0.5*\x,2*\y)}]
        \node [shape=rectangle,draw,fill=green!20, inner sep = 1.5cm] (t5) at (0,0) {};
        \node[bigdot,fill=red] () at (45:1) {};
        \node[bigdot,fill=yellow!90!black] () at (135:1) {};
        \node[bigdot,fill=green!80!black] () at (225:1) {};
        \node[bigdot,fill=blue] () at (315:1) {};

        \node () at (45:1.4) {1};
        \node () at (135:1.4) {2};
        \node () at (225:1.4) {3};
        \node () at (315:1.4) {4};
    \end{scope}
    
    \begin{scope}[shift={(0.5*\x,2*\y)}]
        \node [shape=rectangle,draw,fill=green!20, inner sep = 1.5cm] (t6) at (0,0) {};
        \node[bigdot,fill=yellow!90!black] () at (45:1) {};
        \node[bigdot,fill=red] () at (135:1) {};
        \node[bigdot,fill=blue] () at (225:1) {};
        \node[bigdot,fill=green!80!black] () at (315:1) {};

        \node () at (45:1.4) {2};
        \node () at (135:1.4) {1};
        \node () at (225:1.4) {4};
        \node () at (315:1.4) {3};
    \end{scope}
    
    \begin{scope}[shift={(4.5*\x,2*\y)}]
        \node [shape=rectangle,draw,fill=green!20, inner sep = 1.5cm] (t10) at (0,0) {};
        \node[bigdot,fill=red] (a) at (90:0.8) {};
        \node[bigdot,fill=blue] (b) at (210:0.8) {};
        \node[bigdot,fill=yellow!90!black] (c) at (330:0.8) {};
        \node[bigdot,fill=green!80!black] (d) at (0,0) {};

        \node () at (90:1.2) {1};
        \node () at (210:1.2) {4};
        \node () at (330:1.2) {2};
        \node () at (270:0.4) {3};
    \end{scope}
    
    \begin{scope}[shift={(3.5*\x,2*\y)}]
        \node [shape=rectangle,draw,fill=green!20, inner sep = 1.5cm] (t9) at (0,0) {};
        \node[bigdot,fill=yellow!90!black] (a) at (90:0.8) {};
        \node[bigdot,fill=red] (b) at (210:0.8) {};
        \node[bigdot,fill=green!80!black] (c) at (330:0.8) {};
        \node[bigdot,fill=blue] (d) at (0,0) {};

        \node () at (90:1.2) {2};
        \node () at (210:1.2) {1};
        \node () at (330:1.2) {3};
        \node () at (270:0.4) {4};
    \end{scope}

    \begin{scope}[shift={(2.5*\x,2*\y)}]
        \node [shape=rectangle,draw,fill=green!20, inner sep = 1.5cm] (t8) at (0,0) {};
        \node[bigdot,fill=green!80!black] (a) at (90:0.8) {};
        \node[bigdot,fill=yellow!90!black] (b) at (210:0.8) {};
        \node[bigdot,fill=blue] (c) at (330:0.8) {};
        \node[bigdot,fill=red] (d) at (0,0) {};

        \node () at (90:1.2) {3};
        \node () at (210:1.2) {2};
        \node () at (330:1.2) {4};
        \node () at (270:0.4) {1};
    \end{scope}

    \begin{scope}[shift={(1.5*\x,2*\y)}]
        \node [shape=rectangle,draw,fill=green!20, inner sep = 1.5cm] (t7) at (0,0) {};
        \node[bigdot,fill=blue] (a) at (90:0.8) {};
        \node[bigdot,fill=green!80!black] (b) at (210:0.8) {};
        \node[bigdot,fill=red] (c) at (330:0.8) {};
        \node[bigdot,fill=yellow!90!black] (d) at (0,0) {};

        \node () at (90:1.2) {4};
        \node () at (210:1.2) {3};
        \node () at (330:1.2) {1};
        \node () at (270:0.4) {2};
    \end{scope}


    \begin{scope}[shift={(-3.5*\x,\y)}]
        \node [shape=rectangle,draw,fill=green!20, inner sep = 1.5cm] (m1) at (0,0) {};
        \node[bigdot,fill=red] (d) at (90:0.7) {};
        \node[bigdot,fill=yellow!90!black] (a) at (210:0.7) {};
        \node[bigdot,fill=blue] (b) at (330:0.7) {};
        \node[bigdot,fill=green!80!black] (c) at ($(a)!0.5!(b)$) {};

        \node () at ($(d)+(0,0.5)$) {1};
        \node () at ($(a)+(0,-0.5)$) {2};
        \node () at ($(c)+(0,-0.5)$) {3};
        \node () at ($(b)+(0,-0.5)$) {4};
    \end{scope}

    \begin{scope}[shift={(-2.5*\x,\y)}]
        \node [shape=rectangle,draw,fill=green!20, inner sep = 1.5cm] (m2) at (0,0) {};
        \node[bigdot,fill=yellow!90!black] (d) at (90:0.7) {};
        \node[bigdot,fill=green!80!black] (a) at (210:0.7) {};
        \node[bigdot,fill=red] (b) at (330:0.7) {};
        \node[bigdot,fill=blue] (c) at ($(a)!0.5!(b)$) {};

        \node () at ($(d)+(0,0.5)$) {2};
        \node () at ($(a)+(0,-0.5)$) {3};
        \node () at ($(c)+(0,-0.5)$) {4};
        \node () at ($(b)+(0,-0.5)$) {1};
    \end{scope}
    
    \begin{scope}[shift={(-1.5*\x,\y)}]
        \node [shape=rectangle,draw,fill=green!20, inner sep = 1.5cm] (m3) at (0,0) {};
        \node[bigdot,fill=green!80!black] (d) at (90:0.7) {};
        \node[bigdot,fill=blue] (a) at (210:0.7) {};
        \node[bigdot,fill=yellow!90!black] (b) at (330:0.7) {};
        \node[bigdot,fill=red] (c) at ($(a)!0.5!(b)$) {};

        \node () at ($(d)+(0,0.5)$) {3};
        \node () at ($(a)+(0,-0.5)$) {4};
        \node () at ($(c)+(0,-0.5)$) {1};
        \node () at ($(b)+(0,-0.5)$) {2};
    \end{scope}
    
    \begin{scope}[shift={(-0.5*\x,\y)}]
        \node [shape=rectangle,draw,fill=green!20, inner sep = 1.5cm] (m4) at (0,0) {};
        \node[bigdot,fill=blue] (d) at (90:0.7) {};
        \node[bigdot,fill=red] (a) at (210:0.7) {};
        \node[bigdot,fill=green!80!black] (b) at (330:0.7) {};
        \node[bigdot,fill=yellow!90!black] (c) at ($(a)!0.5!(b)$) {};

        \node () at ($(d)+(0,0.5)$) {4};
        \node () at ($(a)+(0,-0.5)$) {1};
        \node () at ($(c)+(0,-0.5)$) {2};
        \node () at ($(b)+(0,-0.5)$) {3};
    \end{scope}

    \begin{scope}[shift={(3.5*\x,\y)}]
        \node [shape=rectangle,draw,fill=green!20, inner sep = 1.5cm] (m8) at (0,0) {};
        \node[bigdot,fill=red] (d) at (90:0.7) {};
        \node[bigdot,fill=blue] (a) at (210:0.7) {};
        \node[bigdot,fill=yellow!90!black] (b) at (330:0.7) {};
        \node[bigdot,fill=green!80!black] (c) at ($(a)!0.5!(b)$) {};

        \node () at ($(d)+(0,0.5)$) {1};
        \node () at ($(a)+(0,-0.5)$) {4};
        \node () at ($(c)+(0,-0.5)$) {3};
        \node () at ($(b)+(0,-0.5)$) {2};
    \end{scope}

    \begin{scope}[shift={(2.5*\x,\y)}]
        \node [shape=rectangle,draw,fill=green!20, inner sep = 1.5cm] (m7) at (0,0) {};
        \node[bigdot,fill=yellow!90!black] (d) at (90:0.7) {};
        \node[bigdot,fill=red] (a) at (210:0.7) {};
        \node[bigdot,fill=green!80!black] (b) at (330:0.7) {};
        \node[bigdot,fill=blue] (c) at ($(a)!0.5!(b)$) {};

        \node () at ($(d)+(0,0.5)$) {2};
        \node () at ($(a)+(0,-0.5)$) {1};
        \node () at ($(c)+(0,-0.5)$) {4};
        \node () at ($(b)+(0,-0.5)$) {3};
    \end{scope}
    
    \begin{scope}[shift={(1.5*\x,\y)}]
        \node [shape=rectangle,draw,fill=green!20, inner sep = 1.5cm] (m6) at (0,0) {};
        \node[bigdot,fill=green!80!black] (d) at (90:0.7) {};
        \node[bigdot,fill=yellow!90!black] (a) at (210:0.7) {};
        \node[bigdot,fill=blue] (b) at (330:0.7) {};
        \node[bigdot,fill=red] (c) at ($(a)!0.5!(b)$) {};

        \node () at ($(d)+(0,0.5)$) {3};
        \node () at ($(a)+(0,-0.5)$) {2};
        \node () at ($(c)+(0,-0.5)$) {1};
        \node () at ($(b)+(0,-0.5)$) {4};
    \end{scope}
    
    \begin{scope}[shift={(0.5*\x,\y)}]
        \node [shape=rectangle,draw,fill=green!20, inner sep = 1.5cm] (m5) at (0,0) {};
        \node[bigdot,fill=blue] (d) at (90:0.7) {};
        \node[bigdot,fill=green!80!black] (a) at (210:0.7) {};
        \node[bigdot,fill=red] (b) at (330:0.7) {};
        \node[bigdot,fill=yellow!90!black] (c) at ($(a)!0.5!(b)$) {};

        \node () at ($(d)+(0,0.5)$) {4};
        \node () at ($(a)+(0,-0.5)$) {3};
        \node () at ($(c)+(0,-0.5)$) {2};
        \node () at ($(b)+(0,-0.5)$) {1};
    \end{scope}

    
    \begin{scope}[shift={(0,0)}]
        \node [shape=rectangle,draw,fill=green!20, inner sep = 1.5cm] (bottom) at (0,0) {};
        \node[bigdot,fill=red] (a) at (-1.05,0) {};
        \node[bigdot,fill=yellow!90!black] (b) at (-0.35,0) {};
        \node[bigdot,fill=green!80!black] (c) at (0.35,0) {};
        \node[bigdot,fill=blue] (d) at (1.05,0) {};

        \node () at ($(a)+(0,-0.5)$) {1};
        \node () at ($(b)+(0,-0.5)$) {2};
        \node () at ($(c)+(0,-0.5)$) {3};
        \node () at ($(d)+(0,-0.5)$) {4};
    \end{scope}


    \path (bottom.north) 
        edge (m1.south) 
        edge (m2.south) 
        edge (m3.south) 
        edge (m4.south)
        edge (m5.south) 
        edge (m6.south) 
        edge (m7.south) 
        edge (m8.south);

    \path (t5.south)
        edge (m1.north)
        edge (m2.north)
        edge (m3.north)
        edge (m4.north);

    \path (t6.south)
        edge (m5.north)
        edge (m6.north)
        edge (m7.north)
        edge (m8.north);

    \draw (t1.south) -- (m1.north);
    \draw (t2.south) -- (m2.north);
    \draw (t3.south) -- (m3.north);
    \draw (t4.south) -- (m4.north);

    \draw (t7.south) -- (m5.north);
    \draw (t8.south) -- (m6.north);
    \draw (t9.south) -- (m7.north);
    \draw (t10.south) -- (m8.north);

\end{tikzpicture}

%% file: figures/fig-alpha-beta.tex
    \tikzstyle{disk}=[thick,shape=circle,draw,color=black,fill=yellow!10]
    \tikzstyle{dot}=[shape=circle,draw,color=black,fill=black,inner sep=1.5pt]

    \begin{tikzpicture}
        \begin{scope}[shift={(-3,0)}]
            \node[disk, fill = yellow!40, inner sep = 1.2cm] () at (0,0) {};
            \foreach \i in {1,...,5} {
                \coordinate (a\i) at (72*\i+72+90:1.2);
            }
        
            \coordinate (b1) at ($(a4)!0.5!(a5)$);
            \coordinate (b2) at ($(a5)!0.33!(a1)$);
            \coordinate (b3) at ($(a5)!0.67!(a1)$);
            \coordinate (b4) at ($(a2)!0.25!(a3)$);
            \coordinate (b5) at ($(a2)!0.5!(a3)$);
            \coordinate (b6) at ($(a2)!0.75!(a3)$);
            \coordinate (b7) at ($(a3)!0.33!(a4)$);
            \coordinate (b8) at ($(a3)!0.67!(a4)$);

            \draw[BlueLine, line width = 0.5mm] (a1) -- (b8);
        
            \foreach \i in {1,...,5} {
                \draw (a\i) node [dot] {};
            }
        
            \foreach \i in {1,...,8} {
                \draw (b\i) node [dot] {};
            }
        \end{scope}
    
        \begin{scope}[shift={(3,0)}]
            \node[disk, fill = yellow!40, inner sep = 1.2cm] () at (0,0) {};
            \foreach \i in {1,...,5} {
                \coordinate (a\i) at (72*\i+72+90:1.2);
            }
        
            \coordinate (b1) at ($(a4)!0.5!(a5)$);
            \coordinate (b2) at ($(a5)!0.33!(a1)$);
            \coordinate (b3) at ($(a5)!0.67!(a1)$);
            \coordinate (b4) at ($(a2)!0.25!(a3)$);
            \coordinate (b5) at ($(a2)!0.5!(a3)$);
            \coordinate (b6) at ($(a2)!0.75!(a3)$);
            \coordinate (b7) at ($(a3)!0.33!(a4)$);
            \coordinate (b8) at ($(a3)!0.67!(a4)$);

            \filldraw[BluePoly, line width = 0.5mm] (a1) -- (a2) -- (a3) -- (b8) -- cycle;

            \filldraw[BluePoly, line width = 0.5mm] (a4) -- (a5) -- (b3) -- cycle;
            
            \foreach \i in {1,...,5} {
                \draw (a\i) node [dot] {};
            }
        
            \foreach \i in {1,...,8} {
                \draw (b\i) node [dot] {};
            }
        \end{scope}    
    \end{tikzpicture}

%% file: figures/fig-A-B.tex
    \tikzstyle{disk}=[thick,shape=circle,draw,color=black,fill=yellow!10]
    \tikzstyle{dot}=[shape=circle,draw,color=black,fill=black,inner sep=1.5pt]

    \begin{tikzpicture}
        \begin{scope}[shift={(0,0)}]
            \node[disk, fill = yellow!40, inner sep = 1.5cm] () at (0,0) {};
            \foreach \i in {1,...,5} {
                \coordinate (a\i) at (72*\i+72+90:1.5);
            }
        
            \coordinate (b1) at ($(a4)!0.5!(a5)$);
            \coordinate (b2) at ($(a5)!0.33!(a1)$);
            \coordinate (b3) at ($(a5)!0.67!(a1)$);
            \coordinate (b4) at ($(a2)!0.25!(a3)$);
            \coordinate (b5) at ($(a2)!0.5!(a3)$);
            \coordinate (b6) at ($(a2)!0.75!(a3)$);
            \coordinate (b7) at ($(a3)!0.33!(a4)$);
            \coordinate (b8) at ($(a3)!0.67!(a4)$);

            \filldraw[GreenPoly] (a2) -- (a3) -- (b8) -- cycle;
            \filldraw[GreenPoly] (a4) -- (a5) -- (b3) -- cycle;

            \node () at ($(b7)!0.3!(b4)$) {\small $A_{3,2}$};
            \node () at ($(b1)!0.4!(b3)!0.2!(a5)$) {\small $B_{3,2}$};
        
            \foreach \i in {a2,a3,a4,a5,b1,b2,b3,b4,b5,b6,b7,b8} {
                \draw (\i) node [dot] {};
            }
        
            \foreach \i in {a1} {
                \draw (\i) node [opendot,color = gray, fill = white] {};
            }
        \end{scope}    
        
            
        
    \end{tikzpicture}

%% file: figures/fig-tree-slide.tex
    \tikzstyle{RedLine}=[line width=0.3mm,color=red,text=black]
    \tikzstyle{RedPoly}=[RedLine,fill=red!20]
    \tikzstyle{disk}=[thick,shape=circle,draw,color=black,fill=yellow!10]
    \tikzstyle{dot}=[shape=circle,draw,color=black,fill=black,inner sep=1.5pt]
    \newcommand{\makepoints}{
        \coordinate (r) at (30:1cm);
        \coordinate (q) at (150:1cm);
        \coordinate (p) at (270:1cm);
    }
    \newcommand{\drawpoints}{
        \foreach \x in {p,q,r} {
            \draw (\x) node [dot] {};
        }
    }
    \begin{tikzpicture}
        \begin{scope}[shift={(-4,0)}]
            \node[disk, fill = yellow!40, inner sep = 1.2cm] () at (0,0) {};
            \makepoints
            \draw[RedLine, line width = 2] (p) -- (q) -- (r);
            \node () at ($(0,0)!1.4!(p)$) {$p$};
            \node () at ($(0,0)!1.4!(q)$) {$q$};
            \node () at ($(0,0)!1.4!(r)$) {$r$};
            \drawpoints
        \end{scope}
    
        \draw[->, line width = 1.2] (-2,0) -- (2,0);
    
        \begin{scope}[shift={(4,0)}]
            \node[disk, fill = yellow!40, inner sep = 1.2cm] () at (0,0) {};
            \makepoints
            \draw[RedLine, line width = 2] (p) -- (r) -- (q);
            \node () at ($(0,0)!1.4!(p)$) {$p$};
            \node () at ($(0,0)!1.4!(q)$) {$q$};
            \node () at ($(0,0)!1.4!(r)$) {$r$};
            \drawpoints
        \end{scope}    
    \end{tikzpicture}

%% file: figures/fig-thm-boolean.tex
\begin{tikzpicture}

    \begin{scope}[rotate=180]
        \node [disk,fill=yellow!40, inner sep = 3.5cm] () at (0,0) {};
        
        \foreach \i in {1,...,5} {
            \coordinate (a\i) at (72*\i+72+90:3.5);
        }
    
        \coordinate (b1) at ($(a4)!0.5!(a5)$);
        \coordinate (b2) at ($(a5)!0.33!(a1)$);
        \coordinate (b3) at ($(a5)!0.67!(a1)$);
        \coordinate (b4) at ($(a2)!0.25!(a3)$);
        \coordinate (b5) at ($(a2)!0.5!(a3)$);
        \coordinate (b6) at ($(a2)!0.75!(a3)$);
        \coordinate (b7) at ($(a3)!0.33!(a4)$);
        \coordinate (b8) at ($(a3)!0.67!(a4)$);
        \coordinate (b9) at ($(a1)!0.5!(a2)$);

        \begin{scope}[RedLine, line width = 0.7mm]
            \draw (a4) -- (b9);
            \draw (b8) -- (b7) -- (b5) -- (b4) -- (b9) -- (a1) -- (b3) -- (b2) -- (b1);
        \end{scope}

        \begin{scope}[RedLine, line width = 0.3mm, dashed]
            \draw (b1) -- (a5);
            \draw (a2) -- (b9);
            \draw (a3) -- (b6) -- (b7);
        \end{scope}
    
        \foreach \i in {1,...,5} {
            \draw (a\i) node [dot, inner sep = 0.7mm] {};
        }
    
        \foreach \i in {1,...,8,9} {
            \draw (b\i) node [dot, inner sep = 0.7mm] {};
        }

        \foreach \i in {1,...,5} {
            \draw ($(0,0)!1.2!(a\i)$) node () {\large $z_{\i,0}$};
        }

        \draw ($(0,0)!0.8!(b8)$) node () {\large $z_-$};
        \draw ($(0,0)!0.8!(b7)$) node () {\large $v_1$};
        \draw ($(0,0)!0.8!(b5)$) node () {\large $v_2$};
        \draw ($(0,0)!1.2!(b9)$) node () {\large $v_j$};
        \draw ($(0,0)!0.8!(b2)$) node () {\large $v_h$};
        \draw ($(0,0)!0.8!(b1)$) node () {\large $z_+$};
    \end{scope}
\end{tikzpicture}

%% file: figures/fig-boolean-nonexample.tex
    \tikzstyle{disk}=[thick,shape=circle,draw,color=black,fill=yellow!10]
    \tikzstyle{dot}=[shape=circle,draw,color=black,fill=black,inner sep=1.5pt]

    \begin{tikzpicture}
        \begin{scope}[shift={(0,0)}]
            \node[disk, fill = yellow!40, inner sep = 1.5cm] () at (0,0) {};
            \coordinate (a1) at (90:1.5);
            \coordinate (a2) at (210:1.5);
            \coordinate (a3) at (330:1.5);
            \coordinate (b1) at (90:0.75);
            \coordinate (b2) at (210:0.75);
            \coordinate (b3) at (330:0.75);

            \draw[BlueLine, line width = 0.5mm] (a1) -- (b2);
            \draw[BlueLine, line width = 0.5mm] (a2) -- (b3);
            \draw[BlueLine, line width = 0.5mm] (a3) -- (b1);
        
            \foreach \i in {a1,a2,a3,b1,b2,b3} {
                \draw (\i) node [dot] {};
            }
        \end{scope}    
    \end{tikzpicture}